\documentclass[]{amsart}
\usepackage[all,cmtip]{xy}      
\usepackage{mathtools, amsthm, amsfonts}
\usepackage{enumitem}  
\setenumerate{label={\normalfont(\alph*)}}

\newtheorem{theorem}{Theorem}[section]
\newtheorem{lemma}[theorem]{Lemma}
\newtheorem{proposition}[theorem]{Proposition}
\newtheorem{corollary}[theorem]{Corollary}

\theoremstyle{definition}

\newtheorem{remark}[theorem]{Remark}

\begin{document}

\title[The fibre of the degree $3$ Map, Anick spaces and the double suspension]{The fibre of the degree $3$ Map, Anick spaces and the\\ double suspension} 
\author{Steven Amelotte}
\email{steven.amelotte@rochester.edu} 

\subjclass[2010]{55P35, 55P10}
\keywords{loop space decomposition, double suspension, Anick space, Kervaire invariant} 

\begin{abstract}
Let $S^{2n+1}\{p\}$ denote the homotopy fibre of the degree $p$ self map of $S^{2n+1}$. For primes $p \ge 5$, work of Selick shows that $S^{2n+1}\{p\}$ admits a nontrivial loop space decomposition if and only if $n=1$ or $p$. Indecomposability in all but these dimensions was obtained by showing that a nontrivial decomposition of $\Omega S^{2n+1}\{p\}$ implies the existence of a $p$-primary Kervaire invariant one element of order $p$ in~$\pi_{2n(p-1)-2}^S$. We prove the converse of this last implication and observe that the homotopy decomposition problem for $\Omega S^{2n+1}\{p\}$ is equivalent to the strong $p$-primary Kervaire invariant problem for all odd primes. For $p=3$, we use the $3$-primary Kervaire invariant element $\theta_3$ to give a new decomposition of $\Omega S^{55}\{3\}$ analogous to Selick's decomposition of $\Omega S^{2p+1}\{p\}$ and as an application prove two new cases of a long-standing conjecture stating that the fibre of the double suspension $S^{2n-1} \longrightarrow \Omega^2S^{2n+1}$ is homotopy equivalent to the double loop space of Anick's space.
\end{abstract}

\maketitle

\section{Introduction}

Localize all spaces and maps at an odd prime $p$. Let $S^{2n+1}\{p\}$ denote the homotopy fibre of the degree $p$ map on $S^{2n+1}$ and let $W_n$ denote the homotopy fibre of the double suspension $E^2 \colon S^{2n-1} \longrightarrow \Omega^2S^{2n+1}$. In \cite{S Odd} and \cite{S Decomposition}, Selick showed that there is a homotopy decomposition
\begin{equation} \label{Selick's decomposition}
\Omega^2 S^{2p+1}\{p\} \simeq \Omega^2S^3\langle 3 \rangle \times W_p, 
\end{equation}
where $S^3\langle 3 \rangle$ is the $3$-connected cover of $S^3$, and obtained as an immediate corollary that $p$ annihilates all $p$-torsion in $\pi_\ast(S^3)$. This exponent result is generalized by the exponent theorem of Cohen, Moore and Neisendorfer \cite{CMN1, CMN2, N 3-primary}, who used different loop space decompositions to construct a map $\varphi \colon \Omega^2S^{2n+1} \longrightarrow S^{2n-1}$ with the property that the composite \[ \Omega^2S^{2n+1} \stackrel{\varphi}{\longrightarrow} S^{2n-1} \stackrel{E^2}{\longrightarrow} \Omega^2S^{2n+1} \] is homotopic to the $p^\text{th}$ power map on $\Omega^2S^{2n+1}$ and proved by induction on $n$ that $p^n$ annihilates the $p$-torsion in $\pi_\ast(S^{2n+1})$. By a result of Gray \cite{G2}, if $p$ is an odd prime, then $\pi_\ast(S^{2n+1})$ contains infinitely many elements of order $p^n$, so this is the best possible odd primary homotopy exponent for spheres. The work of Cohen, Moore and Neisendorfer suggested that there should exist a space $T^{2n+1}(p)$ fitting in a fibration sequence
\[ \Omega^2S^{2n+1} \stackrel{\varphi}{\longrightarrow} S^{2n-1} \longrightarrow T^{2n+1}(p) \longrightarrow \Omega S^{2n+1} \]
in which their map $\varphi$ occurs as the connecting map. The existence of such a fibration was first proved by Anick for $p \ge 5$ in \cite{An}. A much simpler construction, valid for all odd primes, was later given by Gray and Theriault in \cite{GT2}, in which they also show that Anick's space $T^{2n+1}(p)$ has the structure of an $H$-space and that all maps in the fibration above can be chosen to be $H$-maps. 

A well-known conjecture in unstable homotopy theory states that the fibre $W_n$ of the double suspension $E^2 \colon S^{2n-1} \longrightarrow \Omega^2S^{2n+1}$ is a double loop space. Anick's space represents a potential candidate for a double classifying space of $W_n$, and one of Cohen, Moore and Neisendorfer's remaining open conjectures in \cite{CMN3} states that there should be a $p$-local homotopy equivalence $W_n \simeq \Omega^2T^{2np+1}(p)$. A stronger form of the conjecture (see e.g. \cite{AG}, \cite{G EHP}, \cite{T A case}) states that
\[ BW_n \simeq \Omega T^{2np+1}(p) \]
where $BW_n$ is the classifying space of $W_n$ first constructed by Gray \cite{G}. 
Such equivalences have only been shown to exist for $n=1$ and $n=p$. In the former case, both $BW_1$ and $\Omega T^{2p+1}(p)$ are known to be homotopy equivalent to $\Omega^2S^3\langle 3 \rangle$. Using Anick's fibration, Selick showed in \cite{S Space} that $T^{2p+1}(p) \simeq \Omega S^3\langle 3 \rangle$ and that the decomposition \eqref{Selick's decomposition} can be delooped to a homotopy equivalence
\[ \Omega S^{2p+1}\{p\} \simeq \Omega S^3\langle 3 \rangle \times BW_p. \] 
The $n=p$ case was proved in the strong form $BW_p \simeq \Omega T^{2p^2+1}(p)$ by Theriault \cite{T A case} using the above decomposition in an essential way. Under these identifications, he further showed that $\Omega S^{2p+1}\{p\}$ and $T^{2p+1}(p) \times \Omega T^{2p^2+1}(p)$ are equivalent as $H$-spaces. 

For primes $p \ge 5$, similar decompositions of $\Omega S^{2n+1}\{p\}$ are not possible if $n \neq 1$ or $p$. This result was obtained in \cite{S Reformulation} by first showing that for $n>1$ the existence of a certain spherical homology class imposed by a nontrivial homotopy decomposition of $\Omega S^{2n+1}\{p\}$ implies the existence of an element of $p$-primary Kervaire invariant one in $\pi_{2n(p-1)-2}^S$, and then appealing to Ravenel's \cite{Ra} result on the nonexistence of such elements when $p \ge 5$ and $n\neq p$. For $p=3$, the question of whether $\Omega S^{2n+1}\{3\}$ admits a nontrivial decomposition for $n=3^j$ with $j>1$ was left open. In this short note, we prove that the strong odd primary Kervaire invariant problem is in fact equivalent to the problem of decomposing the loop space $\Omega S^{2n+1}\{p\}$. When $p=3$, this equivalence can be used to import results from stable homotopy theory to obtain new results concerning the unstable homotopy type of $\Omega S^{2n+1}\{3\}$ as well as some cases of the conjecture that $W_n$ is a double loop space.


\begin{theorem} \label{main thm}
Let $p$ be an odd prime. Then the following are equivalent:
\begin{enumerate}
\item There exists a $p$-primary Kervaire invariant one element $\theta_j \in \pi_{2p^j(p-1)-2}^S$ of order $p$;
\item There is a homotopy decomposition of $H$-spaces \[ \Omega S^{2p^j+1}\{p\} \simeq T^{2p^j+1}(p) \times \Omega T^{2p^{j+1}+1}(p). \]
\end{enumerate}
Furthermore, if the above conditions hold, then there are homotopy equivalences of $H$-spaces
\[ BW_{p^{j-1}} \simeq \Omega T^{2p^j+1}(p) \ \textrm{ and } \ BW_{p^j} \simeq \Omega T^{2p^{j+1}+1}(p). \]
\end{theorem}

From this point of view, Selick's decomposition of $\Omega S^{2p+1}\{p\}$ and the previously known equivalences $BW_1 \simeq \Omega T^{2p+1}(p)$ and $BW_p \simeq \Omega T^{2p^2+1}(p)$ correspond to the existence (at all odd primes) of the Kervaire invariant class $\theta_1=\beta_1 \in \pi_{2p^2-2p-2}^S$ given by the first element of the periodic beta family in the stable homotopy groups of spheres. By Ravenel's negative solution to the Kervaire invariant problem for primes $p \ge 5$, Theorem \ref{main thm} has new content only at the prime $p=3$. For example, in addition to the $3$-primary Kervaire invariant element $\theta_1 \in \pi_{10}^S$ for $p=3$ and $j=1$ corresponding to the decomposition of $\Omega S^7\{3\}$, it is known that there exists a $3$-primary Kervaire invariant element $\theta_3 \in \pi_{106}^S$ which we use to obtain the following decomposition of $\Omega S^{55}\{3\}$ and prove the $n=p^2$ and $n=p^3$ cases of the $BW_n \simeq \Omega T^{2np+1}(p)$ conjecture at $p=3$.

\begin{corollary} \label{main cor}
There are $3$-local homotopy equivalences of $H$-spaces
\begin{enumerate}
\item $\Omega S^{55}\{3\} \simeq T^{55}(3) \times \Omega T^{163}(3)$
\item $BW_9 \simeq \Omega T^{55}(3)$
\item $BW_{27} \simeq \Omega T^{163}(3)$.
\end{enumerate}
\end{corollary}

\begin{remark}
The equivalence of conditions (a) and (b) in Theorem 1.1 does not hold for $p=2$. In \cite{CCPS}, Campbell, Cohen, Peterson and Selick showed that for $n>1$ a nontrivial decomposition of the fibre $\Omega^2S^{2n+1}\{2\}$ of the squaring map implies the existence of an element $\theta \in \pi_{2n-2}^S$ of Kervaire invariant one such that $\theta\eta$ is divisible by $2$. Since such elements are well known to exist only for $n=2, 4$ or $8$, these are the only dimensions for which $\Omega^2S^{2n+1}\{2\}$ can decompose nontrivially. Explicit decompositions of $\Omega^2S^5\{2\}$, $\Omega^2S^9\{2\}$ and $\Omega^3S^{17}\{2\}$ corresponding to the first three $2$-primary Kervaire invariant classes $\theta_1=\eta^2$, $\theta_2=\nu^2$ and $\theta_3=\sigma^2$ are given in \cite{C 2-primary}, \cite{CS} and \cite{Am}.
\end{remark}

A further consequence of Theorem \ref{main thm} unique to the $p=3$ case concerns the associativity of Anick spaces. Unlike when $p \ge 5$, in which case $T^{2n+1}(p)$ is a homotopy commutative and homotopy associative $H$-space for all $n\ge 1$, counterexamples to the homotopy associativity of $T^{2n+1}(3)$ have been observed in \cite{G Abelian} and \cite{T Properties}. In particular, in \cite{G Abelian}, it was shown that if $T^{2n+1}(3)$ is homotopy associative, then $n=3^j$ for some $j\ge 0$. The proof given there also shows that a homotopy associative $H$-space structure on $T^{2n+1}(3)$ implies the existence of a three-cell complex 
\[ S^{2n+1} \cup_3 e^{2n+2} \cup e^{6n+1} \] 
with nontrivial mod $3$ Steenrod operation $\mathcal{P}^n$, which in turn implies (by Spanier--Whitehead duality and the Liulevicius--Shimada--Yamanoshita factorization of $\mathcal{P}^{p^j}$ by secondary cohomology operations) the existence of an element of strong Kervaire invariant one. Using Theorem \ref{main thm}, we observe that the converse is also true to obtain the following.

\begin{theorem} \label{second thm}
Let $n > 1$. Then the mod $3$ Anick space $T^{2n+1}(3)$ is homotopy associative if and only if there exists a $3$-primary Kervaire invariant one element of order $3$ in $\pi_{4n-2}^S$.
\end{theorem}

\section{The Proof of Theorem \ref{main thm}}

The bulk of the proof of Theorem \ref{main thm} will consist of a slight generalization of the argument given in \cite[Theorem 1.2]{T A case}, which we briefly describe below. As in \cite{T A case}, the following extension lemma, originally proved in \cite{AG} for $p \ge 5$ and later extended to include the $p=3$ case in \cite{GT2}, will be crucial. We write $P^n(p^r)$ for the mod $p^r$ Moore space $S^{n-1} \cup_{p^r} e^n$ and for a space $X$ define homotopy groups with $\mathbb{Z}/p^r\mathbb{Z}$ coefficients by $\pi_n(X; \mathbb{Z}/p^r\mathbb{Z}) = [P^n(p^r), X]$.

\begin{lemma} \label{extension lemma}  
Let $p$ be an odd prime. Let $X$ be an $H$-space such that $p^k \cdot \pi_{2np^k-1}(X; \mathbb{Z}/p^{k+1}\mathbb{Z})=0$ for $k \ge 1$. Then any map $P^{2n}(p) \to X$ extends to a map $T^{2n+1}(p) \to X$. 
\end{lemma} 


In \cite{GT2}, Anick's space is constructed as the homotopy fibre in a secondary $EHP$ fibration
\begin{equation} \label{fib}
T^{2n+1}(p) \stackrel{E}{\longrightarrow} \Omega S^{2n+1}\{p\} \stackrel{H}{\longrightarrow} BW_n 
\end{equation}
where $E$ is an $H$-map which induces in mod $p$ homology the inclusion of 
\[ H_\ast(T^{2n+1}(p)) \cong \Lambda(a_{2n-1}) \otimes \mathbb{Z}/p\mathbb{Z}[c_{2n}] \]
into
\[ H_\ast(\Omega S^{2n+1}\{p\}) \cong  \left( \bigotimes_{i=0}^\infty \Lambda(a_{2np^i-1}) \right) \otimes \left( \bigotimes_{i=1}^\infty \mathbb{Z}/p\mathbb{Z}[b_{2np^i-2}] \right) \otimes \mathbb{Z}/p\mathbb{Z}[c_{2n}], \]
and $H$ induces the projection onto
\[ H_\ast(BW_n) \cong \left( \bigotimes_{i=1}^\infty \Lambda(a_{2np^i-1}) \right) \otimes \left( \bigotimes_{i=1}^\infty \mathbb{Z}/p\mathbb{Z}[b_{2np^i-2}] \right). \]
When $n=p$, it follows from Selick's decomposition of $\Omega S^{2p+1}\{p\}$ that $H$ admits a right homotopy inverse $s \colon BW_p \to \Omega S^{2p+1}\{p\}$ splitting the homotopy fibration \eqref{fib} in this case. Restricting to the bottom cell of $BW_p$, Theriault \cite{T A case} extended the composite
\[ S^{2p^2-2} \lhook\joinrel\longrightarrow BW_p \stackrel{s}{\longrightarrow} \Omega S^{2p+1}\{p\} \]
to a map $P^{2p^2-1}(p) \to \Omega S^{2p+1}\{p\}$ and then applied Lemma \ref{extension lemma} to the adjoint map $P^{2p^2}(p) \to S^{2p+1}\{p\}$ to obtain an extension $T^{2p^2+1}(p) \to S^{2p+1}\{p\}$. Finally, looping this last map, he showed that the composite
\[ \Omega T^{2p^2+1}(p) \longrightarrow \Omega S^{2p+1}\{p\} \stackrel{H}{\longrightarrow} BW_p \]
is a homotopy equivalence, thus proving the $n=p$ case of the conjecture that $BW_n \simeq \Omega T^{2np+1}(p)$.

In our case, we will use Lemma \ref{extension lemma} to first construct a right homotopy inverse of $H \colon \Omega S^{2n+1}\{p\} \to BW_n$ in dimensions $n=p^j$ for which there exists an element $\theta_j \in \pi_{2p^j(p-1)-2}^S$ of strong Kervaire invariant one and then follow the same strategy as above to obtain both a homotopy decomposition of $\Omega S^{2p^j+1}\{p\}$ and a homotopy equivalence $BW_{p^j} \simeq \Omega T^{2p^{j+1}+1}(p)$. These equivalences can then be used to compare the loops on \eqref{fib} with the $n=p^{j-1}$ case of a homotopy fibration
\[ BW_n \longrightarrow \Omega^2S^{2np+1}\{p\} \longrightarrow W_{np} \]
to further obtain a homotopy equivalence of fibres $BW_{p^{j-1}} \simeq \Omega T^{2p^j+1}(p)$.

\begin{proof}[Proof of Theorem \ref{main thm}]
We first show that condition (b) implies condition (a). Given any homotopy equivalence 
\[ \psi \colon T^{2p^j+1}(p) \times \Omega T^{2p^{j+1}+1}(p) \stackrel{\sim\,}{\longrightarrow} \Omega S^{2p^j+1}\{p\}, \] set $n=p^j$ and let $f$ denote the composite
\[ f \colon S^{2np-2} \lhook\joinrel\longrightarrow \Omega T^{2np+1}(p) \stackrel{i_2}{\longrightarrow} T^{2n+1}(p) \times \Omega T^{2np+1}(p) \stackrel{\psi}{\longrightarrow} \Omega S^{2n+1}\{p\} \]
where the first map is the inclusion of the bottom cell of $\Omega T^{2np+1}(p)$ and the second map $i_2$ is the inclusion of the second factor. Then \[ f_\ast(\iota)=b_{2np-2} \in H_{2np-2}(\Omega S^{2n+1}\{p\}) \]
where $\iota$ is the generator of $H_{2np-2}(S^{2np-2})$. Since the homology class $b_{2np-2}$ is spherical if and only if there exists a stable map $g \colon P^{2n(p-1)-1}(p) \to S^0$ such that the Steenrod operation $\mathcal{P}^n$ acts nontrivially on $H^\ast(C_g)$ by \cite{S Reformulation}, it follows that $\pi_{2n(p-1)-2}^S$ contains an element of $p$-primary Kervaire invariant one and order $p$.

Conversely, suppose there exists a $p$-primary Kervaire invariant one element $\theta_j \in \pi_{2p^j(p-1)-2}^S$ of order $p$. Then by \cite{S Reformulation}, the homology class $b_{2p^{j+1}-2} \in H_{2p^{j+1}-2}(\Omega S^{2p^j+1}\{p\})$ is spherical, so there exists a map $f \colon S^{2p^{j+1}-2} \to \Omega S^{2p^j+1}\{p\}$ with Hurewicz image $b_{2p^{j+1}-2}$. Now following the proof of \cite[Theorem 1.2]{T A case}, since $\Omega S^{2p^j+1}\{p\}$ has $H$-space exponent $p$, it follows that $f$ has order $p$ and hence extends to a map 
\[ e \colon P^{2p^{j+1}-1}(p) \longrightarrow \Omega S^{2p^j+1}\{p\}. \] 
Let $\hat{e} \colon P^{2p^{j+1}}(p) \to S^{2p^j+1}\{p\}$ denote the adjoint of $e$. Again, because $\Omega S^{2p^j+1}\{p\}$ has $H$-space exponent $p$, we have that \[ p\cdot \pi_\ast(S^{2p^j+1}\{p\}; \mathbb{Z}/p^k\mathbb{Z})=0 \] for all $k \ge 1$, and since $S^{2p^j+1}\{p\}$ is an $H$-space \cite{N1}, the map $\hat{e}$ satisfies the hypotheses of Lemma \ref{extension lemma} and therefore admits an extension 
\[ s \colon T^{2p^{j+1}+1}(p) \longrightarrow S^{2p^j+1}\{p\}. \] 
Note that this factorization of $\hat{e}$ through $s$ implies that the adjoint map $e$ factors through $\Omega s$, so we have a commutative diagram
\[
\xymatrix{
S^{2p^{j+1}-2} \ar[r] \ar[ddrr]^f & P^{2p^{j+1}-1}(p) \ar[r] \ar[ddr]^e & \Omega T^{2p^{j+1}+1}(p) \ar[dd]^{\Omega s} \\
&&\\
& & \Omega S^{2p^j+1}\{p\}
}
\]
where the maps along the top row are skeletal inclusions, and hence $(\Omega s)_\ast$ is an isomorphism on $H_{2p^{j+1}-2}(\;)$ since $f_\ast$ is. Now since $H \colon \Omega S^{2p^j+1}\{p\} \to BW_{p^j}$ induces an epimorphism in homology, the composite
\[ \Omega T^{2p^{j+1}+1}(p) \stackrel{\Omega s\,}{\longrightarrow} \Omega S^{2p^j+1}\{p\} \stackrel{H}{\longrightarrow} BW_{p^j} \]
induces an isomorphism of the lowest nonvanishing reduced homology group 
\[ H_{2p^{j+1}-2}(\Omega T^{2p^{j+1}+1}(p)) \cong H_{2p^{j+1}-2}(BW_{p^j}) \cong \mathbb{Z}/p\mathbb{Z}. \]
By \cite{GT1}, any map $\Omega T^{2np+1}(p) \to BW_{n}$ which is degree one on the bottom cell must be a homotopy equivalence, and thus $H \circ \Omega s$ is a homotopy equivalence. Composing a homotopy inverse of $H \circ \Omega s$ with $\Omega s$, we obtain a right homotopy inverse of $H$, which shows that the homotopy fibration
\[ T^{2p^j+1}(p) \stackrel{E}{\longrightarrow} \Omega S^{2p^j+1}\{p\} \stackrel{H}{\longrightarrow} BW_{p^j} \]
splits. Moreover, letting $m$ denote the loop multiplication on $\Omega S^{2p^j+1}\{p\}$, the composite
\[ T^{2p^j+1}(p) \times \Omega T^{2p^{j+1}+1}(p) \xrightarrow{E \times \Omega s\,} \Omega S^{2p^j+1}\{p\} \times \Omega S^{2p^j+1}\{p\} \stackrel{m\,}{\longrightarrow} \Omega S^{2p^j+1}\{p\} \]
defines an equivalence of $H$-spaces since $E$ and $\Omega s$ are $H$-maps and $m$ is homotopic to the loops on the $H$-space multiplication on $S^{2p^j+1}\{p\}$.

The homotopy equivalence $H \circ \Omega s \colon \Omega T^{2p^{j+1}+1}(p) \to BW_{p^j}$ is not necessarily multiplicative, but the $H$-space decomposition of $\Omega S^{2p^j+1}\{p\}$ constructed above can now be used exactly as in the proof of \cite[Theorem 1.1]{T A case} to produce an $H$-map $BW_{p^j} \to \Omega T^{2p^{j+1}+1}(p)$ which is also a homotopy equivalence. 

It remains to show that there is an equivalence of $H$-spaces $BW_{p^{j-1}} \simeq \Omega T^{2p^j+1}(p)$. In his construction of a classifying space of $W_n$, Gray \cite{G} introduced a $p$-local homotopy fibration
\[ BW_n \stackrel{j}{\longrightarrow} \Omega^2S^{2np+1} \stackrel{\phi}{\longrightarrow} S^{2np-1} \]
where the map $j$ has order $p$ and hence lifts to a map $j' \colon BW_n \to \Omega^2S^{2np+1}\{p\}$. By \cite{T 3-primary}, $j'$ can be chosen to be an $H$-map when $p\ge 3$. Since $j_\ast$ is an isomorphism in degree $2np-1$, it follows by commutativity with the Bockstein that $j'_\ast$ is an isomorphism in degree $2np-2$. Let $\gamma$ denote the equivalence of $H$-spaces $T^{2p^j+1}(p) \times \Omega T^{2p^{j+1}+1}(p) \stackrel{\sim\,}{\longrightarrow} \Omega S^{2p^j+1}\{p\}$ constructed above. As $\Omega\gamma$ is also an equivalence of $H$-spaces, it has a homotopy inverse $(\Omega\gamma)^{-1}$ which is also an $H$-map. Consider the composite
\[ BW_{p^{j-1}} \stackrel{j'}{\longrightarrow} \Omega^2S^{2p^j+1}\{p\} \xrightarrow{(\Omega\gamma)^{-1}} \Omega T^{2p^j+1}(p) \times \Omega^2T^{2p^{j+1}+1}(p) \stackrel{\pi_1\,}{\longrightarrow} \Omega T^{2p^j+1}(p) \]
where $j'$ is the lift of $j$ with $n=p^{j-1}$ and $\pi_1$ is the projection onto the first factor. Since all three maps in this composition induce isomorphisms on $H_{2p^j-2}(\;)$, it again follows from the atomicity result in \cite{GT1} that the composite defines a homotopy equivalence $BW_{p^{j-1}} \simeq \Omega T^{2p^j+1}(p)$, which is an equivalence of $H$-spaces since each map above is an $H$-map.
\end{proof}

\section{Applications}

In this section we derive Corollary \ref{main cor} and Theorem \ref{second thm} from Theorem \ref{main thm} and discuss some other consequences in the $p=3$ case.

\subsection{The homotopy decomposition of $\Omega S^{55}\{3\}$}

Since, by \cite{S Reformulation}, $\Omega S^{2n+1}\{p\}$ is atomic for all $n$ such that $\pi_{2n(p-1)-2}^S$ contains no element of $p$-primary Kervaire invariant one, $\Omega S^{2n+1}\{p\}$ is indecomposable for $n \neq p^j$ and it follows from Theorem \ref{main thm} that the decomposition problem for $\Omega S^{2n+1}\{p\}$ is equivalent to the strong $p$-primary Kervaire invariant problem for odd primes $p$. The $3$-primary Kervaire invariant problem is open, but the elements $b_{j-1} \in \mathrm{Ext}_{\mathcal{A}_p}^{2,2p^j(p-1)}(\mathbb{F}_p, \mathbb{F}_p)$ in the $E_2$-term of the Adams spectral sequence which potentially detect elements of odd primary Kervaire invariant one are known to behave differently for $p=3$ than they do for primes $p \ge 5$. 

While $b_0$ is a permanent cycle representing $\theta_1 \in \pi_{2p(p-1)-2}^S$ at all odd primes, Ravenel showed in \cite{Ra} that for $j>1$ and $p \ge 5$ the elements $b_{j-1}$ support nontrivial differentials in the Adams spectral sequence and hence that none of the $\theta_j$ exist for $j>1$ and $p \ge 5$. For $p=3$, however, it is known (see \cite{Ra, Ra Green}) that, although $b_1$ supports a nontrivial differential, $b_2$ is a permanent cycle representing a $3$-primary Kervaire invariant class $\theta_3 \in \pi_{106}^S$. 


\begin{proof}[Proof of Corollary \ref{main cor}]
According to \cite{Ra Green}, $\pi_{106}^S \cong \mathbb{Z}/3\mathbb{Z}$ after localizing at $p=3$, so $\theta_3$ has order $3$ and the result follows from Theorem \ref{main thm}. 
\end{proof}

\begin{remark}
We note that the nonexistence of $\theta_2$ at $p=3$ implies that $\Omega S^{2p^2+1}\{p\} = \Omega S^{19}\{3\}$ is atomic and hence indecomposable by the result in \cite{S Reformulation} mentioned above.
\end{remark}

Observe that since the mod $p$ Moore space $P^2(p)$ is the homotopy cofibre of the degree $p$ self map $p \colon S^1 \to S^1$, by applying the functor $\mathrm{Map}_\ast(-, S^{2n+1})$ to the homotopy cofibration 
\[ S^1 \stackrel{p}{\longrightarrow} S^1 \longrightarrow P^2(p) \] 
we obtain a homotopy fibration
\[ \mathrm{Map}_\ast(P^2(p), S^{2n+1}) \longrightarrow \Omega S^{2n+1} \stackrel{p}{\longrightarrow} \Omega S^{2n+1} \]
which identifies the mapping space $\mathrm{Map}_\ast(P^2(p), S^{2n+1})$ with the homotopy fibre $\Omega S^{2n+1}\{p\}$ of the $p^\mathrm{th}$ power map on the loop space $\Omega S^{2n+1}$. The decomposition of $\Omega S^{55}\{3\}$ in Corollary \ref{main cor} therefore induces the following splitting of homotopy groups with $\mathbb{Z}/3\mathbb{Z}$ coefficients analogous to Selick's \cite{S Decomposition} splitting of $\pi_\ast(S^{2p+1}; \mathbb{Z}/p\mathbb{Z})$.

\begin{corollary}
For $k\ge 4$, there are isomorphisms
\begin{align*}
\pi_k(S^{55}; \mathbb{Z}/3\mathbb{Z}) &\cong \pi_{k-2}(T^{55}(3)) \oplus \pi_{k-1}(T^{163}(3)) \\
&\cong \pi_{k-4}(W_9) \oplus \pi_{k-3}(W_{27}).
\end{align*}
\end{corollary}

\subsection{Homotopy associativity and exponents for mod $3$ Anick spaces}

The following two useful properties of $T^{2n+1}(p)$ were conjectured by Anick and Gray \cite{An, AG}:
\begin{enumerate}[label=(\roman*)]
\item $T^{2n+1}(p)$ is a homotopy commutative and homotopy associative $H$-space; \label{prop1}
\item $T^{2n+1}(p)$ has homotopy exponent $p$. \label{prop2}
\end{enumerate}
Both properties have been established for all $p \ge 5$ and $n \ge 1$, but only partial results have been obtained in the $p=3$ case. For example, it was found in \cite{T Properties} that $T^7(3)$ is both homotopy commutative and homotopy associative but that homotopy associativity fails for $T^{11}(3)$. More generally, Gray showed in \cite{G Abelian} that if $T^{2n+1}(3)$ is homotopy associative, then $n=3^j$ for some $j\ge 0$ and moreover that property (i) implies property (ii).

Concerning property (ii), in general $T^{2n+1}(3)$ is only known to have homotopy exponent bounded above by $9$. (This can be seen using fibration \eqref{fib} and the fact that $BW_n$ has $3$-primary exponent $3$, for example.) Since $T^{2n+1}(p)$ is an $H$-space for all $p\ge 3$, one could also ask for the stronger property that $T^{2n+1}(p)$ has $H$-space exponent $p$, i.e., that its $p^\mathrm{th}$ power map is null homotopic. We note that, when they occur, decompositions of $\Omega S^{2n+1}\{3\}$ give some evidence for (ii).
\pagebreak

\begin{corollary}
The following hold:
\begin{enumerate}
\item $T^7(3)$ and $T^{55}(3)$ are homotopy commutative and homotopy associative $H$-spaces;
\item $T^7(3)$, $T^{55}(3)$, $\Omega T^{19}(3)$ and $\Omega T^{163}(3)$ each have $H$-space exponent $3$.
\end{enumerate}
\end{corollary}

\begin{proof}
Since the homotopy equivalences $\Omega S^7\{3\} \simeq T^7(3) \times \Omega T^{19}(3)$ and $\Omega S^{55}\{3\} \simeq T^{55}(3) \times \Omega T^{163}(3)$ which follow from Theorem \ref{main thm} are equivalences of $H$-spaces, part (b) follows immediately from the fact that $\Omega S^{2n+1}\{3\}$ has $H$-space exponent $3$ \cite{N1}, and part (a) follows from the fact that $\Omega S^{2n+1}\{3\}$ is homotopy associative and homotopy commutative as it is the loop space of an $H$-space.
\end{proof}

\begin{proof}[Proof of Theorem \ref{second thm}]
Let $n>1$ and suppose $T^{2n+1}(3)$ is homotopy associative. Then the proof of \cite[Theorem A.2]{G Abelian} shows that there exists a three-cell complex 
\[ X = S^{2n+1} \cup_3 e^{2n+2} \cup e^{6n+1} \] 
with nontrivial mod $3$ Steenrod operation $\mathcal{P}^n \colon H^{2n+1}(X) \to H^{6n+1}(X)$. The attaching map of the middle cell of a Spanier--Whitehead dual of $X$ then defines an element of Kervaire invariant one in $\pi_{4n-2}^S$ which has order $3$ since it extends over a mod $3$ Moore space. Alternatively, by \cite[Proposition 7.1]{T Properties}, the homotopy associativity of $T^{2n+1}(3)$ implies that a certain composite
\[ S^{6n-3} \xrightarrow{[\iota,[\iota,\iota]]\,} \Omega S^{2n} \stackrel{r}{\longrightarrow} S^{2n-1} \]
is divisible by $3$, where $[\iota,[\iota,\iota]]$ denotes the triple Samelson product of the generator of $\pi_{2n-1}(\Omega S^{2n})$ and $r$ is a left homotopy inverse of the suspension $E \colon S^{2n-1} \to \Omega S^{2n}$. It is easy to check that the composite above coincides with the image of the generator of the lowest nonvanishing $3$-local homotopy group $\pi_{6n-3}(W_n) \cong \mathbb{Z}/3\mathbb{Z}$ under the homotopy fibre map $W_n \to S^{2n-1}$, and the divisibility of this element is a well-known equivalent formulation of the strong Kervaire invariant problem.

Conversely, if there exists a $3$-primary Kervaire invariant element of order $3$ in $\pi_{4n-2}^S$, then $n=3^j$ for some $j\ge 1$ and it follows from Theorem \ref{main thm} that $T^{2n+1}(3)$ is a homotopy associative $H$-space as it is an $H$-space retract of a loop space.
\end{proof}


\section{A stable splitting of $\Omega S^{2n+1}\{p\}$}

It is well known that $S^{2n+1}\{p\}$ splits as a wedge of mod $p$ Moore spaces after suspending once. In this section we determine the stable homotopy type of the loop space $\Omega S^{2n+1}\{p\}$ by observing that although the homotopy fibration
\[ T^{2n+1}(p) \stackrel{E}{\longrightarrow} \Omega S^{2n+1}\{p\} \stackrel{H}{\longrightarrow} BW_n \]
only splits in Kervaire invariant one dimensions, it splits for all $n$ after suspending twice. As in the previous sections, $p$ denotes an odd prime and all spaces and maps are assumed to be localized at $p$.

\begin{proposition} \label{stable splitting}
For all $n\ge 1$, there is a homotopy equivalence
\[ \Sigma^2\Omega S^{2n+1}\{p\} \simeq \Sigma^2(T^{2n+1}(p) \times BW_n). \]
\end{proposition}

\begin{proof}
In \cite{G}, Gray showed that the classifying space $BW_n$ of the fibre of the double suspension fits in a homotopy fibration 
\[ S^{2n-1} \stackrel{E^2}{\longrightarrow} \Omega^2S^{2n+1} \stackrel{\nu}{\longrightarrow} BW_n \]
and that there is a homotopy equivalence $\Sigma^2\Omega^2S^{2n+1} \simeq \Sigma^2(S^{2n-1} \times BW_n).$ Let $s \colon \Sigma^2BW_n \to \Sigma^2\Omega^2S^{2n+1}$ be a right homotopy inverse of $\Sigma^2\nu$. In the construction of Anick's fibration in \cite{GT2}, $T^{2n+1}$ is defined as the homotopy fibre of the map $H$, where $H$ is constructed as an extension 
\[
\xymatrix{
\Omega^2S^{2n+1} \ar[r]^-\partial \ar[d]^\nu & \Omega S^{2n+1}\{p\} \ar@{-->}[dl]^H \\
BW_n
}
\]
of $\nu$ through the connecting map of the homotopy fibration $\Omega S^{2n+1}\{p\} \longrightarrow \Omega S^{2n+1} \stackrel{p}{\longrightarrow} \Omega S^{2n+1}$. Therefore, by composing $s$ with $\Sigma^2\partial$, we obtain a right homotopy inverse $s' \colon \Sigma^2BW_n \to \Sigma^2\Omega S^{2n+1}\{p\}$ of $\Sigma^2H$. Next, consider the composite map $f$ defined by
\[ f\colon T^{2n+1}(p) \wedge \Sigma^2BW_n \xrightarrow{E\wedge s'} \Omega S^{2n+1}\{p\} \wedge \Sigma^2\Omega S^{2n+1}\{p\} \longrightarrow \Sigma^2\Omega S^{2n+1}\{p\} \]
where the second map is obtained by suspending the Hopf construction $\Sigma\Omega S^{2n+1}\{p\} \wedge \Omega S^{2n+1}\{p\} \to \Sigma\Omega S^{2n+1}\{p\}$ on $\Omega S^{2n+1}\{p\}$. Finally, since $\Sigma^2E$, $s'$ and $f$ each induce monomorphisms in mod $p$ homology, it follows that the map
\[ \Sigma^2(T^{2n+1}(p) \times BW_n) \simeq \Sigma^2T^{2n+1}(p) \vee \Sigma^2BW_n \vee (\Sigma^2T^{2n+1}(p) \wedge BW_n) \longrightarrow \Sigma^2S^{2n+1}\{p\} \]
defined by their wedge sum is a homology isomorphism and hence a homotopy equivalence.
\end{proof}

It follows from Proposition \ref{stable splitting} that $\Omega S^{2n+1}\{p\}$ has the stable homotopy type of a wedge of Moore spaces, Snaith summands $D_{2,k}(S^{2n-1})$ of the stable splitting of $\Omega^2S^{2n+1}$, and their smash products.

A similar argument can be used to give a stable splitting of the homotopy fibre $E^{2n+1}$ of the natural inclusion $i \colon P^{2n+1}(p) \to S^{2n+1}\{p\}$ where $BW_n$ is a stable retract. More precisely, it follows from \cite{GT2} that the extension $H$ of $\nu$ appearing in the proof of Proposition \ref{stable splitting} can be chosen to factor through a map $\delta \colon \Omega S^{2n+1}\{p\} \to E^{2n+1}$, and thus $BW_n$ also retracts off $E^{2n+1}$ after suspending twice. The space $E^{2n+1}$ along with a homotopy pullback diagram
\begin{equation} \label{diagram}
\begin{gathered}
\xymatrix{
\Omega^2S^{2n+1} \ar[r] \ar[d] & E^{2n+1} \ar[r] \ar@{=}[d] & F^{2n+1} \ar[r] \ar[d] & \Omega S^{2n+1} \ar[d] \\
\Omega S^{2n+1}\{p\} \ar[r]^-\delta & E^{2n+1} \ar[r] & P^{2n+1}(p) \ar[r]^-i \ar[d]^q & S^{2n+1}\{p\} \ar[d] \\
& & S^{2n+1} \ar@{=}[r] & S^{2n+1}.
}
\end{gathered}
\end{equation}
determined by the factorization of the pinch map $q \colon P^{2n+1}(p) \to S^{2n+1}$ in the bottom right square was thoroughly analyzed in Cohen, Moore and Neisendorfer's study of the homotopy theory of Moore spaces \cite{CMN1, CMN2}, where decompositions of $\Omega E^{2n+1}$, $\Omega F^{2n+1}$ and $\Omega P^{2n+1}(p)$ were used to determine the homotopy exponents of spheres and Moore spaces. The double suspension fibration $W_n \longrightarrow S^{2n-1} \stackrel{E^2}{\longrightarrow} \Omega^2S^{2n+1}$ was shown to retract off the homotopy fibration along the top row of the loops on \eqref{diagram} with $W_n$ and $S^{2n-1}$ appearing as factors containing the bottom cells in product decompositions of $\Omega E^{2n+1}$ and $\Omega F^{2n+1}$, respectively. 

Consider the morphism of homotopy fibrations 
\begin{equation} \label{diagram2}
\begin{gathered}
\xymatrix{
T^{2n+1}(p) \ar[r] \ar[d]^E & X \ar[d] \\
\Omega S^{2n+1}\{p\} \ar[r]^-\delta \ar[d]^H & E^{2n+1} \ar[d] \\
BW_n \ar@{=}[r] & BW_n
}
\end{gathered}
\end{equation}
determined by the factorization of $H$ through $\delta$. As with the splitting in Proposition \ref{stable splitting}, the retraction of $\Sigma^2BW_n$ off $\Sigma^2E^{2n+1}$ can be desuspended in Kervaire invariant one dimensions. One difference, however, is that since $W_n$ is always a retract of $\Omega E^{2n+1}$ by Cohen, Moore and Neisendorfer's decomposition, the image of the homology class $b_{2np-2} \in H_{2np-2}(\Omega S^{2n+1}\{p\})$ under $\delta_\ast$ is spherical for all $n$ (as opposed to just those $n=p^j$ for which $\pi_{2n(p-1)-2}^S$ contains an element of strong Kervaire invariant one), and thus the nonexistence of Kervaire invariant elements does not obstruct the possibility of an unstable decomposition of $E^{2n+1}$ as it does for $\Omega S^{2n+1}\{p\}$. It would therefore be interesting to know if a homotopy class $S^{2np-2} \to E^{2n+1}$ with Hurewicz image $\delta_\ast(b_{2np-2})$ could be extended to a map $\Omega T^{2np+1}(p) \to E^{2n+1}$ as in the proof of Theorem \ref{main thm} to prove the conjectured homotopy equivalence $BW_n \simeq \Omega T^{2np+1}(p)$ for all $n$ and split the homotopy fibration in the second column of the diagram above. We show below that $BW_n$ is in fact a retract of $E^{2n+1}$ for all $n$, delooping the result of Cohen, Moore and Neisendorfer.

\begin{proposition}
For all $n\ge 1$, $BW_n$ is a retract of $E^{2n+1}$.
\end{proposition}

\begin{proof}
By the construction of $BW_n$ in \cite{G}, there is a homotopy fibration sequence 
\[ \Omega^2S^{2n+1} \longrightarrow BW_n \times S^{4n-1} \longrightarrow S^{2n} \stackrel{E}{\longrightarrow} \Omega S^{2n+1} \]
where the connecting map factors as $\Omega^2S^{2n+1} \stackrel{\nu}{\longrightarrow} BW_n \stackrel{i_1}{\longrightarrow} BW_n \times S^{4n-1}$. Since the composite $S^{2n} \stackrel{E}{\longrightarrow} \Omega S^{2n+1} \longrightarrow S^{2n+1}\{p\}$ is just the inclusion of the bottom cell of $S^{2n+1}\{p\}$, there is a homotopy commutative diagram
\[
\xymatrix{
BW_n \times S^{4n-1} \ar[r] \ar[d] & E^{2n+1} \ar[d] \\
S^{2n} \ar[r] \ar[d]^E & P^{2n+1}(p) \ar[d]^i \\
\Omega S^{2n+1} \ar[r] & S^{2n+1}\{p\}
}
\]
where the induced map of fibres determines a map $g \colon BW_n \to E^{2n+1}$. Observe that the connecting map $\Omega^2S^{2n+1} \to BW_n \times S^{4n-1}$ of the first column induces an isomorphism on $H_{2np-2}(\;)$ since $\nu$ does, and the connecting map $\delta \colon \Omega S^{2n+1}\{p\} \to E^{2n+1}$ of the second column induces an isomorphism on $H_{2np-2}(\;)$ by the commutativity of \eqref{diagram2}. Therefore, since the map $\Omega^2S^{2n+1} \to \Omega S^{2n+1}\{p\}$ given by the loops on the bottom horizontal map induces a monomorphism in homology by a Serre spectral sequence argument, we conclude that $g \colon BW_n \to E^{2n+1}$ induces an isomorphism on $H_{2np-2}(\;)$ so that the composition $BW_n \xrightarrow{g} E^{2n+1} \to BW_n$ with the extension in \eqref{diagram2} is degree one on the bottom cell and thus a homotopy equivalence.
\end{proof}

\end{document}